\documentclass[12pt]{amsart}

\usepackage{amsaddr}
\usepackage{amscd,amssymb}

\title[Poisson's fundamental theorem of calculus]{Poisson's fundamental theorem of calculus via Taylor's formula}

\theoremstyle{theorem}
\newtheorem*{theorem}{Theorem}
\theoremstyle{definition}
\newtheorem*{definition}{Definition}
\newtheorem*{example}{Example}
\newtheorem*{remark}{Remark}

\author{Patrik Nystedt}
\address{University West,
Department of Engineering Science, 
SE-461 86 Trollh\"{a}ttan, Sweden}

\email{patrik.nystedt@hv.se}

\begin{document}

\maketitle

\begin{abstract}
We use Taylor's formula with Lagrange remainder to make a modern adaptation of Poisson's proof 
of a version of the fundamental theorem of calculus in the case when the integral is defined by Euler sums,
that is Riemann sums with left (or right) endpoints which are equally spaced.
We discuss potential benefits for such an approach in basic calculus courses.
\end{abstract}

\section{Introduction}

Historians \cite[p. 235]{edwards1979} argue that Leibniz traced back his 
inspiration for the cal\-culus and, in particular, for the result that nowadays
is called the fundamental theorem of calculus (FTC), from a simple observation concerning sequences
he made in his youth.
Namely, if we are given a sequence $F_0,F_1,F_2,\ldots,F_n$ and we define
the sequence of differences $f_1 = F_1-F_0, f_2 = F_2-F_1, \ldots, f_n = F_n - F_{n-1}$,
then $\sum_{k=1}^n f_k = F_n - F_0$. 
In other words, the sum of the differences of the consecutive terms in a sequence
equals the difference between the last and the first term.
Leibniz then applied this observation
to the setting of a function $F$ with derivative $f$ and considered an infinite sum
of infinitesimals so that $\int_a^b dF = F(b)-F(a)$ or, in other words,
the version of the FTC presented in most calculus books today (see e.g. \cite{adams2006,hass2017,stewart2015}), namely
\begin{equation*}
\int_a^b f(x) \ dx = F(b) - F(a).
\end{equation*}
From a modern point of view, Leibniz's argument is of course incomplete,
since an infinite sum of infinitesimals makes no sense (unless one
uses Robinson's non-standard analysis \cite{robinson1966}).
The aim of this note is to show that if we define the integral as a limit of Euler sums,
i.e. Riemann sums with equally spaced intervals with left endpoints,
then we can adapt an argument of Poisson to easily prove a useful enough 
version of the FTC via Taylor's formula with Lagrange
remainder based on Leibniz's simple
observation concerning sums of differences.
The proof also yields an error approximation for the corresponding Euler sums.
It seems to the author of the present article that this set up for the
FTC is not so well known and even, in some sense, neglected, in favor for the highly abstract integral defined by general Riemann sums, 
but that it would be suitable for introductory calculus courses thanks to its simplicity.

\section{The integral}

In an introductory calculus course we wish to present the FTC in a reasonably rigorous fashion
and at the same time not making our definitions and proofs too abstract.
As many teachers have experienced, this is not so easy.
Indeed, there is ample evidence 
(see e.g. \cite{jones2013,orton1983,sealey2014,thompson1994}) that
students have great difficulties understanding the meaning
of the integral and, even more so, comprehending why the FTC holds.
It is, however, not hard to imagine why this is so since the integral has gradually developed as a concept over
a period of thousands of years
and the modern textbook version of the Darboux integral \cite{darboux1875} as defined by upper and lower Riemann sums 
is, after all, only roughly 150 years old. All of the earlier versions of the integral have, from a modern point of view, flaws.
Indeed, Riemann himself, in his seminal paper \cite{riemann1867}, neglects to point out 
the completeness property of the real numbers.
Cauchy \cite{cauchy1823}, upon whose ideas Riemann based his contribution, only used left endpoint
subdivisions of the interval in his definition of the integral.
Going back even further, as we mentioned above, Leibniz considered the integral 
to be sums of infinitesimals, or, in his own words in a letter to l'Hospital:
''...to find quadratures is nothing else but to find sums, provided that one supposes that the
differences are incomparably small'' \cite[p. 245]{edwards1979}.
Before that, Cavalieri \cite[p. 104]{edwards1979} in his work
''Geometry of indivisibles'' regarded an area
as consisting of parallel and equidistant line segments.

So how should we define the integral in an introductory calculus course?
Well, one problem with the set up
using general Riemann sums is that the associated integral then is not a limit
in the usual sense of {\it sequences}, that our students are used to, but rather a limit of a {\it net} \cite{olmstead1961}.
Not only is such a definition unsuitable for concrete calculations, for instance 
using computer simulations, but also highly abstract.
A lot of teachers will, nonetheless, first define general Riemann sums but then 
almost immediately say that they only will be using equally spaced intervals.
Then, in concrete examples, they will specialize this even further by using left (or right) endpoint intervals. 
We suggest that the integral, from a pedagogical perspective, should be defined in this way.
This puts as somewhere in between Leibniz's infinitesimals and Riemann's formal definition.
Since this definition, when interpreted as a method to calculate inte\-grals, coincides with an 
algorithm that Euler \cite[Part I, Section I, Chapter 7]{euler1768} suggested for calculating approximations of integrals,
we will henceforth call such sums ''Euler sums'', to distinguish them from general Riemann sums.
In the sequel, we let $\mathbb{Z}$ denote the integers and we let $\mathbb{N}$ denote the set 
of natural numbers $\{ 1,2,3,\ldots \}$.

\begin{definition}\label{defintegral}
Suppose that $f$ is a real valued function defined on an interval $[a,b]$.
For all $n \in \mathbb{N}$ and all $k \in \mathbb{Z}$, we put $\Delta x = (b-a)/n$
and $x_k = a + k \Delta x$.
We say that $I_n = \sum_{k=0}^{n-1} f( x_k ) \Delta x$ is the $n^{\rm th}$ {\it Euler sum} of $f$ on $[a,b]$
and we say that $f$ is {\it integrable} on $[a,b]$
if the limit $I = \lim_{n \to \infty} I_n$ exists.
In that case, we call $I$ the {\it integral} of $f$ on $[a,b]$ and we write
this symbolically as  $\int_a^b f(x) \ dx = I$.
\end{definition}

The above definition is mathematically crystal clear and
the Euler sums are easy for students to calculate in particular cases.

\begin{example}
If $f(x) = x$, $a=0$ and $b=1$, then, using a simple computer program, the students will find that 
$I_1 = 0$, $I_{10} = 0.45$, $I_{100} = 0.495$, $I_{1000} = 0.4995$ and $I_{10000} = 0.49995$,
which suggests that $I = 1/2$.
A teacher (or a clever student) can then verify this hypothesis using limits of arithmetic sums
in the standard formal way 
\begin{eqnarray*}
\int_0^1 x \ dx & = & \lim_{n \to \infty} I_n = \lim_{n \to \infty} \sum_{k=0}^{n-1} \frac{k}{n} \cdot \frac{1}{n} 
= \lim_{n \to \infty} \frac{1}{n^2} \sum_{k=0}^{n-1} k \\
                   & = & \lim_{n \to \infty} \frac{(n-1)n}{2 n^2} = \lim_{n \to \infty} \frac{1}{2} - \frac{1}{2n} = \frac{1}{2}.
\end{eqnarray*}
\end{example}

\section{Poisson's fundamental theorem of calculus}

The first person to use the adjective ''fundamental'' for a result related
to the FTC was, according to \cite{bressoud2011}, the French mathematician Poisson.
In fact, Poisson refers to the equation $\int f(x) \ dx = F(b)-F(a)$ as 
''the fundamental proposition of the theory of definite integrals'' (see loc. cit.).
Interestingly, Poisson seems also to be the first person to give a formal
proof of a version of the FTC in the case when the sums are of Euler's type.
In his proof Poisson assumes that a primitive function exists
and that it has a power series expansion over the interval (see \cite[p. 192]{grabiner1983} or 
Poisson's original paper \cite{poisson1820}).
Investigating Poisson's proof in detail reveals that the underlying idea 
is to approxi\-mate the function by a straight line, carefully
keeping track of the error, that is Taylor's formula with Lagrange remainder.
We now state and prove a modern version of the FTC, using Euler sums, which is based on Poisson's ideas,
assuming minimal required regularity of the functions involved for the proof to go through.

\begin{theorem}\label{thmftc}
If $F$ is a real valued function defined on $[a,b]$ such that its first derivative
exists and is continuous on $[a,b]$, and its second derivative exists and is bounded on $(a,b)$,
then $f=F'$ is integrable on $[a,b]$ and $\int_a^b f(x) dx = F(b)-F(a).$ 
\end{theorem}

\begin{proof} 
We use the notation introduced earlier.
From Taylor's formula with Lagrange remainder (see e.g. \cite{adams2006,hass2017,olmstead1961} or the next section), 
it follows that $$F( x_{k+1} ) = F( x_k ) + f( x_k ) \Delta x + f'(c) \frac{(\Delta x)^2}{2}$$
for some $c \in (x_k , x_{k+1})$, depending on $k$ and $\Delta x$,
for $k \in \{ 0,\ldots,n-1 \}$. Re\-arranging this equality yields that
$$f( x_k ) \Delta x = F( x_{k+1} ) - F( x_k ) - f'(c) \frac{(\Delta x)^2}{2}.$$
Using Leibniz's observation concerning sums of differences, we therefore get that
\begin{eqnarray*}
\int_a^b f(x) \ dx & = & 
\lim_{n \to \infty} \sum_{k=0}^{n-1} f( x_k ) \Delta x \\
& = & \lim_{n \to \infty } \left[ \sum_{k=0}^{n-1} F( x_{k+1} ) - F( x_k ) - \frac{(\Delta x)^2}{2} \sum_{k=0}^{n-1} f'(c) \right] \\
& = & F(b)-F(a) + \lim_{n \to \infty} \frac{(b-a)^2}{2n^2} \sum_{k=0}^{n-1} f'(c) \\
& = & F(b) - F(a)
\end{eqnarray*}
since the triangle inequality implies that
\begin{eqnarray*}
\lim_{n \to \infty} \left| \frac{(b-a)^2}{2n^2} \sum_{k=0}^{n-1} f'(c) \right| \leq 
\lim_{n \to \infty} \frac{M(b-a)^2}{2n} = 0
\end{eqnarray*}
for any $M$ satisfying $|f'(x)| \leq M$ when $a < x < b$.
\end{proof}

\begin{remark}
From the above proof, we immediately get the error bound 
$$|I - I_n| \leq \frac{M(b-a)^2}{2n},$$
for all $n \in \mathbb{N}$,
where $M = {\rm sup} \{ \ |f'(x)| \ ; \ a < x < b \ \}$,
for the $n^{\rm th}$ Euler sum.
\end{remark}

\section{Taylor's formula}

The experienced teacher might argue that Poisson's proof of the FTC is short
since it uses Taylor's formula with Lagrange remainder, a result that not always is proved in introductory
calculus courses because of its difficulty.
For instance, in \cite{adams2006} a 
rather complicated proof by induction using a generalized mean value theorem is provided.
In \cite{hass2017} a proof is given which uses Rolle's theorem repeatedly.
In \cite{stewart2015} the Lagrange form of the remainder term is not obtained at all, but is instead
replaced by an inequality which, however, is proved {\it using} the FTC.
We propose the following short and elegant proof which only uses Rolle's theorem once.
Note that the idea for this proof certainly is not new. 
Indeed, it can be found in one of G. H. Hardy's books \cite{hardy1908} and also in several 
popular textbooks from the 1960:s (see e.g. \cite{olmstead1961}).

\begin{theorem}\label{thmtaylor}
Let $n$ be a non-negative integer. 
If $f$ is a real-valued function defined on $[a,b]$ such that its $n^{th}$ derivative exists,
is continuous on $[a,b]$, and is differentiable on $(a,b)$,
then there exists $c \in (a,b)$ such that
$$f(b) = \sum_{j=0}^n \frac{ f^{(j)}(a) }{ j! } (b-a)^i + \frac{ f^{(n+1)}(c) }{ (n+1)! } (b-a)^{n+1}.$$
\end{theorem}

\begin{proof}
There is a real number $R$ such that
$$f(b) = \sum_{j=0}^n \frac{ f^{(j)}(a) }{ j! } (b-a)^j + R (b-a)^{n+1}.$$
To find this $R$, we use an auxiliary function,
$$g(x) = f(b) - \sum_{i=0}^n \frac{ f^{(i)}(x) }{ j! } (b-x)^j - R (b-x)^{n+1}.$$
Then, clearly, $g(a) = g(b) = 0$. Rolle's theorem implies that there exists $c \in (a,b)$ satisfying 
$$0 = g'(c) = -\frac{ f^{(n+1)}(c) }{ n! } (b-c)^n + R(n+1)(b-c)^n.$$
Solving for $R$ gives us $R = f^{n+1}(c)/(n+1)!$
to conclude the proof.
\end{proof}

\section{Discussion}

In this article we have presented a version of the FTC using a version of the 
integral defined by limits of Euler sums.
We feel that such an approach would be beneficial for students studying calculus 
for the first time, for many reasons.
First of all, the students can, using a computer program, easily calculate the Euler sums
to formulate hypotheses concerning integrals, for instance the values of $\int_0^1 x^n \ dx$ for $n \in \mathbb{N}$.
Secondly, as we have shown above, it is easy to prove the FTC using Euler sums.
Indeed, it is easy to compare our approach to the FTC with the complexity of the classical one
using general Riemann sums. For instance, in Tao's book \cite[p. 306-342]{tao2006} it takes more than 30 pages
of dense mathematics to rigorously define Riemann sums and to prove the FTC,
not including the fact that continuous functions on compact intervals are uniformly continuous.
No wonder, we as teachers have to cheat somewhere taking this path.
In our simplified approach it us takes one page of simple calculations to reach our goal - without any cheating.
Also, our proof uses Leibniz's simple observation on sums of differences.
Therefore, a teacher can set up a lecture concerning the FTC by first explaining this relation in an informal
way using infinitesimals, as we did in the introduction, and then, in the formal proof, point out exactly where
Leibniz's observation is used. Leibniz was not wrong, he simply did not have 
the tools to make a careful analysis of the theorem.

There are, of course, down sides to our suggested approach.
Some students might ask why we consider left intervals, and not right intervals?
Or even midpoint intervals? If that happens, or if you raise this question yourself 
in a lecture, then this is a good oppor\-tunity to ask, does this really matter?
That is, if we would have defined the integral in either of these ways, do we get the
same set of integrable functions? If so, do these integrals have the same value?
For the functions satisfying the regularity conditions in our version of the FTC 
it is not hard to see that the left/right/midpoint approaches yield the same integrals, using variants of the same proof.
This might even be a good exercise for interested students.
At this point in the course, the teacher could mention general Riemann sums
and explain that a more complicated proof, but in fact based on more or less
Leibniz's idea, would give us an FTC that holds for continuous functions,
and that the primitives in this case would have to be constructed via the integral.

Another rather obvious objection to our version of the FTC is that
it does not hold for functions $f$ with unlimited derivative on the interval.
So for instance, we would not be able to calculate $\int_0^1 3 \sqrt{x} \ dx$
using our FTC. However, again, this raises some interesting questions.
For instance, is the function $f(x) = 3 \sqrt{x}$ integrable on $[0,1]$? Using computer calculations,
the students would get $I_{10} = 1.83153$, $I_{100} = 1.98438$, $I_{1000} = 1,99848$
and $I_{10000} = 1.99985$, which strongly indicates that it is integrable and that $\int_0^1 3 \sqrt{x} \ dx = 2$.
The teacher might then make the valid calcu\-lation $\int_{\epsilon}^1 3 \sqrt{x} \ dx =
2 - 2 \epsilon^{3/2}$, for $\epsilon > 0$, and ask if this might be used somehow
to find the value of the integral over the interval $[0,1]$ using a limiting procedure?
Once we have treated the case with unbounded derivatives, we are naturally lead
to other types of generalized integrals such as unbounded functions
and integrals over unbounded inter\-vals.

Another common objection to our approach is that the integral concept ought to be defined
so that the equation $\int_a^b f(x) \ dx + \int_b^c f(x) \ dx = \int_a^c f(x) \ dx$ always holds for
integrable functions $f(x)$. Now, if we let $1_{\mathbb{Q}}$ denotes the {\it indicator function}
on $[0,1]$, that is $1_{\mathbb{Q}}(x) = 1$, if $x \in \mathbb{Q} \cap [0,1]$, and
$1_{\mathbb{Q}}(x) = 0$, otherwise, then it follows that 
$\int_0^1 1_{\mathbb{Q}}(x) \ dx = 1$ but 
$\int_0^{\sqrt{2}-1} 1_{\mathbb{Q}}(x) \ dx + 
\int_{\sqrt{2}-1}^1 1_{\mathbb{Q}}(x) \ dx = 0$,
since $\sqrt{2}-1$ is an irrational number.
On the other hand, this example does not speak in favour of the Riemann integral,
simply because $1_{\mathbb{Q}}$ is {\it not} Riemann integrable.
However, $1_{\mathbb{Q}}$ is Lebesgue integrable with integral equal to zero
(see for instance \cite[p. 264]{apostol1974}).

\section*{acknowledgment}
The author thanks Jacques G\'{e}linas for providing refe\-rences 
to published material after reading a first version of this article,
including \cite{thompson1989} where the Taylor's formula is proved using 
the FTC, which makes these two results equivalent.

\vfill\eject


\begin{thebibliography}{1}

\bibitem{adams2006}
R. A. Adams, {\it Calculus a complete course.} Toronto: Addison Wesley (2006).

\bibitem{apostol1974}
T. M. Apostol,
{\it Mathematical analysis}.
Addison-Wesley. Second edition (1974).

\bibitem{bressoud2011}
D. M. Bressoud,
Historical Reflections on Teaching the Fundamental Theorem of Integral calculus,
The American Mathematical Monthly, {\bf 118}:2, 99--115 (2011).

\bibitem{cauchy1823}
A.-L. Cauchy, 
{\it Resume des lecons donnees a L'Ecole Royale Poly technique sur Le
calcul infinitesimal, Oeuvres}, Ser. 2, Vol. 4. Paris: Gauthier-Villars,
Tome premier (1823).

\bibitem{darboux1875}
G. Darboux, 
Memoire sur la theorie des functions discontinues.
{\it Ann Sci Ecole Norm Sup} {\bf 4}(2) 57--112 (1875).

\bibitem{edwards1979}
C. H. Edwards Jr.,
{\it The Historical Development of the Calculus},
Springer-Verlag, New York (1979).

\bibitem{euler1768}
L. Euler,
{\it Institutionum calculi integralis} (1768). Note that: 

http://www.17centurymaths.com/ 
provides a complete English translation of this work by Ian Bruce.

\bibitem{grabiner1983}
J. V. Grabiner,
Who Gave You the Epsilon? The Origins of Cauchy's Rigorous Calculus.
{\it The American Mathematical Monthly} {\bf 90}(3) 185--194 (1983)

\bibitem{hardy1908}
G. H. Hardy,
{\it A Course of Pure Mathematics},
The English Language Book Society,
Cambridge University Press (1908). 

\bibitem{hass2017}
J. R. Hass, C. E. Heil and M. D. Weir,
{\it Thomas' Calculus},
Pearson, 14 edition (2017).

\bibitem{jones2013}
S. Jones, Understanding the integral: Students' symbolic forms. 
{\it Journal of Mathematical Behaviour} {\bf 32} 122--141 (2013).

\bibitem{olmstead1961}
J. H. Olmstead, {\it Advanced Calculus},
New York: Appleton-Century Crofts (1961).

\bibitem{orton1983}
A. Orton, Students' understanding of integration. 
{\it Educational Studies in Mathematics}, {\bf 14}(1), 1--18 (1983).

\bibitem{poisson1820}
S.-D. Poisson.
Suite de m\'{e}moire sur les int\'{e}grales d\'{e}finies,
{\it Journal de L'\'{E}cole Royale Polytechnique} {\bf 11}, 295--335 (1820).

\bibitem{riemann1867}
B. Riemann,
{\it \"{U}ber die Darstellbarkeit einer Function durch eine trigonometrische Reihe},
G\"{o}ttingen (1867).

\bibitem{robinson1966}
A. Robinson, 
{\it Non-standard Analysis.} 
Amsterdam, London: North-Holland (1966).

\bibitem{sealey2014}
V. Sealey,
{\it A framework for characterizing student understanding ofRiemann sums and definite integrals}
{\bf 33} 230--245 (2014).

\bibitem{stewart2015}
J. Stewart,
{\it Calculus: Early Transcendentals},
Cengage Learning; 8th ed. edition (2015).

\bibitem{tao2006}
T. Tao,
{\it Analysis I},
Hindustan Book Agency, New Dehli (2006)

\bibitem{thompson1989}
H. B. Thompson,
Taylor's Theorem Using the Generalized Riemann Integral.
The American Math. Monthly,
{\bf 96}, No. 4, pp. 346--350 (1989).

\bibitem{thompson1994}
P. Thompson, 
Images of rate and operational understanding of the fundamental theorem of calculus. 
{\it Educational Studies in Mathematics}, {\bf 26}, 229--274 (1994).

\end{thebibliography}
\end{document}